\documentclass[a4paper, 12pt]{article}

\usepackage[utf8]{inputenc}
\usepackage[T1]{fontenc}
\usepackage[a4paper, margin=2.5cm]{geometry}
\usepackage{needspace}

\usepackage{libertine} 
\usepackage{inconsolata} 

\usepackage{amsmath, amsthm, amssymb}
\usepackage{graphicx}
\usepackage[shortlabels]{enumitem}
\usepackage{authblk}
\usepackage{thmtools}
\usepackage{thm-restate}
\usepackage[colorlinks=true, citecolor=red]{hyperref}
\usepackage{float}
\usepackage[]{algorithm2e}
\usepackage{tikz}
\usetikzlibrary{arrows,shapes,automata,backgrounds,petri}

\renewenvironment{abstract}
{\small\vspace{-1em}
\begin{center}
\bfseries\abstractname\vspace{-.5em}\vspace{0pt}
\end{center}
\list{}{
\setlength{\leftmargin}{0.6in}%
\setlength{\rightmargin}{\leftmargin}}%
\item\relax}
{\endlist}

\declaretheorem[name=Theorem, numberwithin=section]{theorem}

\declaretheorem[name=Question, style=remark, sibling=theorem]{question}
\def\cqedsymbol{\ifmmode$\lrcorner$\else{\unskip\nobreak\hfil
\penalty50\hskip1em\null\nobreak\hfil$\lrcorner$
\parfillskip=0pt\finalhyphendemerits=0\endgraf}\fi}

\def\BZ{\textcolor{black!20!blue}{B}}
\def\RZ{\textcolor{black!20!red}{R}}
\def\SH{\textcolor{black!20!red}{\mathcal{S}_H}}
\def\SG{\textcolor{black!20!blue}{\mathcal{S}_G}}
\def\dB{\textcolor{black!20!blue}{d_H}}
\def\dR{\textcolor{black!20!red}{d_G}}

\let\leq\leqslant
\let\geq\geqslant

\title{A note on deterministic zombies\thanks{V.B. and M.B. are supported by 
the ANR Project GrR (\textsc{ANR-18-CE40-0032}), and L.B. and M.B. are 
supported by the ANR Project DISTANCIA (\textsc{ANR-17-CE40-0015}).}}

\author[1]{Valentin Bartier}
\author[2]{Laurine Bénéteau}
\author[3]{Marthe Bonamy}
\author[4]{Hoang La}
\author[3]{Jonathan Narboni}
\affil[1]{G-SCOP, Univ. Grenoble Alpes, CNRS, Grenoble, France.}
\affil[2]{Aix-Marseille Université, CNRS, Universit\'e de Toulon, LIS 
Marseille, France}
\affil[3]{LaBRI, Université de Bordeaux, CNRS, Bordeaux, France.}
\affil[4]{LIRMM, Université de Montpellier, CNRS, Montpellier, France.}
\date{\today}

\begin{document}

\maketitle

\begin{abstract}
``Zombies and Survivor'' is a variant of the well-studied game of ``Cops and 
Robber'' where the zombies (cops) can only move closer to the survivor 
(robber). We consider the deterministic version of the game where a zombie can 
choose their path if multiple options are available. The zombie number, like 
the cop number, of a graph is the minimum number of zombies, or cops, required 
to capture the survivor. In this short note, we solve a question by Fitzpatrick 
et al., proving that the zombie number of the Cartesian product of two graphs 
is at most the sum of their zombie numbers. We also give a simple graph family 
with cop number $2$ and an arbitrarily large zombie number.
\end{abstract}


\section{Introduction}\label{sec:intro}

We consider a pursuit-and-evasion game defined 
in~\cite{fitzpatrick2016deterministic} as follows: ``Zombies and Survivors is a 
new variant of the well-studied game of Cops and Robber, in which zombies take 
the place of the cops and survivor take the place of the robber. The zombies, 
being of limited intelligence, have a very simple objective in each round -- to 
move closer to a survivor. Therefore, each zombie must move along some shortest 
path, or geodesic, joining itself and a nearest survivor. We say that the 
zombies capture a survivor if one of the zombies moves onto the same vertex as 
a survivor.'' In this version, zombies may have a choice as to which shortest 
path to follow, if there are multiple ones. A different version of the game 
involves randomness in the choice of the shortest path. We refer the interested 
reader to~\cite{bonato2016probabilistic,pralat2019many} (more generally 
to~\cite{bonato2017graph} for a nice survey around cops and robbers) and do not 
consider the topic further.

Following~\cite{fitzpatrick2016deterministic}, we only consider the case of a 
unique survivor in the graph, and assume all graphs throughout the paper to be 
connected. We denote by $z(G)$ the minimum number of zombies to place around a 
graph $G$ so as to ensure that the survivor will eventually be captured. 
Similarly, we denote by $c(G)$ the minimum number of cops to place around a 
graph $G$ so as to ensure that the robber will eventually be captured.

We focus on the following two questions:

\begin{question}[Question 10 
in~\cite{fitzpatrick2016deterministic}]\label{qu:cartesian}
Is $z(G \square H)\leq z(G)+z(H)$ for all graphs $G$ and $H$?
\end{question}

\begin{question}[Question 19 
in~\cite{fitzpatrick2016deterministic}]\label{qu:largeratio}
Over all graphs $G$, how large can the ratio $\frac{z(G)}{c(G)}$ be?
\end{question}

Here, we answer Question~\ref{qu:cartesian} in the affirmative, improving upon 
Theorems 11, 13 and 14 in~\cite{fitzpatrick2016deterministic}. By noting that 
$z(Q_3)=2$, we also obtain immediately that $z(Q_n)=\lceil \frac{2n}3 \rceil$. 
This was the object of Conjecture 18 in~\cite{fitzpatrick2016deterministic}, 
though it was since solved independently in~\cite{offner2019comparing} 
and~\cite{fitzpatrick2018game}.

\begin{theorem}\label{th:cartproduct}
For all graphs $G$ and $H$, we have $z(G \square H)\leq z(G)+z(H)$.
\end{theorem}

We also argue that the ratio in Question~\ref{qu:largeratio} can be arbitrarily 
large. This was already argued in~\cite{offner2019comparing}, but our 
construction and proof are arguably simpler. Additionally, the graphs we 
present are outerplanar graphs, and in fact cacti. Informally, this gives 
little hope for Question~\ref{qu:largeratio} to have a bounded answer in a 
meaningful graph class.

\begin{theorem}\label{th:petals}
For every integer $k$, there is a graph $G_k$ that is a union of cycles sharing 
a vertex such that $z(G_k)\geq k$.
\end{theorem}

We prove Theorem~\ref{th:cartproduct} in Section~\ref{sec:fewaszombias}, 
Theorem~\ref{th:petals} in Section~\ref{sec:muchaszombias}, and conclude in 
Section~\ref{sec:ccl} with some open problems which seem of interest to us.

\section{Cartesian products of graphs}\label{sec:fewaszombias}

\begin{proof}[Proof of Theorem~\ref{th:cartproduct}]

Given a vertex $u \in G \square H$, we denote its coordinates in $G$ and $H$ as 
$(u_G,u_H)$. Given two vertices $u,v$ in $G\square H$, we denote 
$d_G(u,v)=d_G(u_G,v_G)$ (resp. $d_H(u,v)=d_H(u_H,v_H)$) the distance between 
$u$ and $v$ in the projection of $G\square H$ on $G$ (resp. $H$). A \emph{copy} 
of $G$ (resp. $H$) is the subgraph induced in $G \square H$ by all vertices $u$ 
with $u_H=w$ (resp. $u_G=x$) where $w$ is some vertex in $H$ (resp. $x$ is some 
vertex in $G$). Let $\SG$ be an optimal strategy for $z(G)$ zombies in $G$, and 
$\SH$ be an optimal strategy for $z(H)$ zombies in $H$. Throughout the proof, 
we denote by $s$ the vertex where the survivor lies.

We are now ready to describe a winning strategy (for zombies) involving 
$z(G)+z(H)$ zombies. We will distinguish two types of zombies: a set $\BZ$ of 
$z(G)$ \textcolor{black!20!blue}{blue} zombies, which are placed according to 
$\SG$ in some copy of $G$, and a set $\RZ$ of $z(H)$ 
\textcolor{black!20!red}{red} zombies, which are placed according to $\SH$ in 
some copy of $H$. Note that for every $x,y \in \BZ$, we have 
$d_H(x,s)=d_H(y,s)$. We maintain that property step after step, and denote the 
corresponding value $\dB$. Similarly, for every $x,y \in \RZ$, we have 
$d_G(x,s)=d_G(y,s)$: we denote that value $\dR$.

The set $\BZ$ applies the following strategy: as long as $\dB$ is positive, all 
the zombies in $\BZ$ move towards $s$ in $H$ (choosing to keep the same 
coordinate in $G$). Note that this is a valid move, as there is a shortest path 
to $s$ going through the corresponding vertex. Once $\dB=0$, all zombies in 
$\BZ$ either follow $\SG$ (if $s_H$ is unchanged) or move toward $s$ in $H$ to 
remain in the same copy of $G$ as $s$ (if $s_H$ changed). Note that either way, 
we maintain $\dB=0$. The set $\RZ$ applies the same strategy, symmetrically 
with $H$ and $G$ instead of $G$ and $H$. 

We observe that neither $\dB$ nor $\dR$ increases. Additionally, at every step, 
either $s_H$ or $s_G$ is unchanged. Assume $s_H$ is unchanged. Then $\dB$, if 
positive, decreases. If $\dB=0$, then all zombies in $\BZ$ follows $\SG$. Since 
$\dB=0$ for the rest of the game, $\BZ$ is one step closer to catching the 
survivor. Meanwhile, if $s_H$ is changed, then $\dB$ does not change, and $\BZ$ 
is not further away from capturing the survivor according to $\SG$. Since the 
winning strategy $\SG$ terminates in a finite number of steps, and the same 
analysis holds for $\SH$, the process for $G \square H$ terminates and the 
survivor is eventually captured and eaten. 

\end{proof}

\vspace{-0.9cm}

\section{Following the Busan strategy}\label{sec:muchaszombias}

\begin{proof}[Proof of Theorem~\ref{th:petals}]
For $k \in \mathbb{N}^*$, let $G_{k}$ be the graph obtained by taking $k$ 
disjoint copies of $C_5,C_{13},\ldots,C_{2^{k+2}-3}$, for a total of $k^2$ 
cycles, and merging all of them on one vertex $u$ (see 
Figure~\ref{fig:cycles}). Note that $|V(G_k)|\sim k\cdot 2^{k+3}$ as $k 
\rightarrow \infty$. We will argue that $z(G_k)\geq k$. We define a direction 
for all cycles, which we will refer to as \emph{clockwise}.

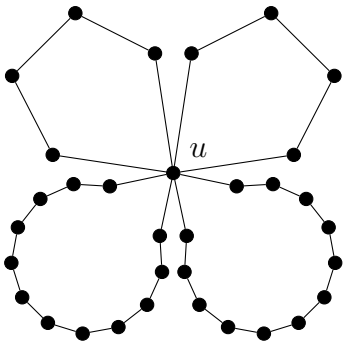
\begin{figure}[!h]
\center
\begin{tikzpicture}
    \tikzstyle{whitenode}=[draw,circle,fill=white,minimum size=9pt,inner 
    sep=0pt]
    \tikzstyle{blacknode}=[draw,circle,fill=black,minimum size=5pt,inner 
    sep=0pt]
    \tikzstyle{rednode}=[draw,circle,fill=red,minimum size=9pt,inner sep=0pt]
    \tikzstyle{greennode}=[draw,circle,fill=black!20!green,minimum 
    size=9pt,inner sep=0pt]

\draw (0,0) node[blacknode][label=45:$u$] (a) {};
\draw (a) ++(45:1.6cm) node (r1) {};
\draw (a) ++(135:1.6cm) node (r2) {};
\draw (a) ++(-135:1.6cm) node (r3) {};

\draw (r1) ++(-135:1cm) node (a1) {};
\draw (r1) ++(-135+72:1cm) node[blacknode] (b1) {};
\draw (r1) ++(-135+2*72:1cm) node[blacknode] (c1) {};
\draw (r1) ++(-135+3*72:1cm) node[blacknode] (d1) {};
\draw (r1) ++(-135+4*72:1cm) node[blacknode] (e1) {};

\draw (a) -- (b1) -- (c1) -- (d1) -- (e1) -- (a);

\draw (r2) ++(-45:1cm) node (a2) {};
\draw (r2) ++(-45+72:1cm) node[blacknode] (b2) {};
\draw (r2) ++(-45+2*72:1cm) node[blacknode] (c2) {};
\draw (r2) ++(-45+3*72:1cm) node[blacknode] (d2) {};
\draw (r2) ++(-45+4*72:1cm) node[blacknode] (e2) {};

\draw (a) -- (b2) -- (c2) -- (d2) -- (e2) -- (a);

\draw (r3) ++(45:1cm) node (a3) {};
\draw (r3) ++(45+27.6923:1cm) node[blacknode] (b3) {};
\draw (r3) ++(45+2*27.6923:1cm) node[blacknode] (c3) {};
\draw (r3) ++(45+3*27.6923:1cm) node[blacknode] (d3) {};
\draw (r3) ++(45+4*27.6923:1cm) node[blacknode] (e3) {};
\draw (r3) ++(45+5*27.6923:1cm) node[blacknode] (f3) {};
\draw (r3) ++(45+6*27.6923:1cm) node[blacknode] (g3) {};
\draw (r3) ++(45+7*27.6923:1cm) node[blacknode] (h3) {};
\draw (r3) ++(45+8*27.6923:1cm) node[blacknode] (i3) {};
\draw (r3) ++(45+9*27.6923:1cm) node[blacknode] (j3) {};
\draw (r3) ++(45+10*27.6923:1cm) node[blacknode] (k3) {};
\draw (r3) ++(45+11*27.6923:1cm) node[blacknode] (l3) {};
\draw (r3) ++(45+12*27.6923:1cm) node[blacknode] (m3) {};

\draw (a) -- (b3) -- (c3) -- (d3) -- (e3) -- (f3) -- (g3) -- (h3) -- (i3) -- 
(j3) -- (k3) -- (l3) -- (m3) -- (a);

\draw (a) ++(-45:1.6cm) node (r3) {};

\draw (r3) ++(135:1cm) node (a3) {};
\draw (r3) ++(135+27.6923:1cm) node[blacknode] (b3) {};
\draw (r3) ++(135+2*27.6923:1cm) node[blacknode] (c3) {};
\draw (r3) ++(135+3*27.6923:1cm) node[blacknode] (d3) {};
\draw (r3) ++(135+4*27.6923:1cm) node[blacknode] (e3) {};
\draw (r3) ++(135+5*27.6923:1cm) node[blacknode] (f3) {};
\draw (r3) ++(135+6*27.6923:1cm) node[blacknode] (g3) {};
\draw (r3) ++(135+7*27.6923:1cm) node[blacknode] (h3) {};
\draw (r3) ++(135+8*27.6923:1cm) node[blacknode] (i3) {};
\draw (r3) ++(135+9*27.6923:1cm) node[blacknode] (j3) {};
\draw (r3) ++(135+10*27.6923:1cm) node[blacknode] (k3) {};
\draw (r3) ++(135+11*27.6923:1cm) node[blacknode] (l3) {};
\draw (r3) ++(135+12*27.6923:1cm) node[blacknode] (m3) {};

\draw (a) -- (b3) -- (c3) -- (d3) -- (e3) -- (f3) -- (g3) -- (h3) -- (i3) -- 
(j3) -- (k3) -- (l3) -- (m3) -- (a);

\end{tikzpicture}\caption{The graph $G_2$}\label{fig:cycles}
\end{figure}

Assume for a contradiction that $z(G_k)\leq k-1$, and let $z_1,\ldots,z_{k-1}$ 
be an initial position of zombies in $G_k$ for a winning strategy. Since there 
are $k$ copies of cycles $C_5,C_{13},\ldots,C_{2^{k+2}-3}$, and only $k-1$ 
zombies, by the pigeon-hole principle there is one copy which contains no $z_i$ 
except possibly for $u$. We will focus on $u$ and the vertices of that copy, 
and ignore from now on the rest of the graph. The goal, perhaps somewhat 
counter-intuitively, is to gather zombies closely behind the survivor, so that 
eventually the survivor can safely circle around the cycle of length 
$2^{k+2}-3$ forever without encountering any zombie. \emph{Circling around} a 
cycle means walking around the cycle clockwise until reaching $u$.

\begin{algorithm}[H]\label{alg:surviving}
\For{$i$ from $1$ to $k$}{
 \While{the $i^\textrm{th}$ closest zombie is at distance at least $2^{i+2}-1$}{
 if the survivor has not chosen a starting point yet, choose the second vertex 
 in the cycle of length $2^{i+2}-3$\;
  circle around the cycle of length $2^{i+2}-3$\;
 }}
 \caption{A winning strategy for the survivor}
\end{algorithm}

The strategy for the survivor is elementary (see 
Algorithm~\ref{alg:surviving}). By circling around in an appropriate way, the 
survivor makes sure that at some point, the first $i$ zombies are within 
distance $2^{i+2}-2$ behind. Since there are only $k-1$ zombies, this 
guarantees that circling around the cycle of length $2^{k+2}-3$ is eventually 
safe and leads to a surviving strategy for the survivor. The only crucial 
property about the behaviour of zombies is that the distance between the 
survivor and a given zombie never increases. Note that since all cycles are 
odd, free will has in fact no impact for zombies.

In Algorithm~\ref{alg:surviving}, zombies are ranked by increasing distance to 
the survivor, with ties broken arbitrarily. The $k^\textrm{th}$ zombie, which 
does not exist, is considered to be at infinite distance. When the survivor has 
not chosen a starting point yet, distance is considered as distance to $u$ 
(which might not be a suitable starting point as there could be a zombie on it).

To argue that Algorithm~\ref{alg:surviving} is safe for the survivor, it 
suffices to point out that when the survivor enters the cycle of length 
$2^{i+2}-3$ (for some $i$), all zombies are either at distance at most 
$2^{i+1}-2$ or at least $2^{i+2}-1$. In the first case, the shortest path to 
the survivors makes them circle around the cycle clockwise (since 
$2^{i+1}-2<\frac{2^{i+2}-3}2$). In the second case, they do not reach $u$ 
before the survivor has finished circling around the cycle (since 
$2^{i+2}-1>2^{i+2}-3+1$).

\end{proof}

\vspace{-0.9cm}


\section{Conclusion}\label{sec:ccl}

To conclude, we offer two open questions. While not of obvious depth, we 
believe that both touch at the heart of what it means for a graph $G$ to 
require $z(G)$ zombies. In particular, if a survivor plays so as to survive for 
as long as possible, are all $z(G)$ zombies within short distance at time of 
death?

\begin{question}\label{qu:poiluszombius}
For every graph $G$, and for a graph $G'$ obtained from $G$ by successively 
adding vertices of degree $1$, does it always hold that $z(G')= z(G)$?
\end{question}

Question~\ref{qu:poiluszombius} can be interpreted as: is there any advantage 
for zombies to individually wait for some pre-announced time at the beginning 
of the game (and then activate and follow the standard rules)?

\begin{question}\label{qu:petaluszombius}
For any graph $G$, is there an integer $k$ such that, for $G'_k$ the graph 
obtained from $G$ by subdividing all edges $k$ times then adding the original 
edges, $z(G'_k)\geq z(G)+1$?
\end{question}

As far as we can tell, it could be that $5$ is a valid answer to 
Question~\ref{qu:petaluszombius} in any graph.

\paragraph*{Addendum.}

After submission of this manuscript, an independent proof of 
Theorem~\ref{th:cartproduct} was published~\cite{keramatipour2021zombie}.

\paragraph*{Acknowledgements.}

We gratefully acknowledge support from the Simon family for the organisation of 
the $6^{\textrm{th}}$ Pessac Graph Workshop, where this research was done. Last 
but not least, we thank Peppie for her unwavering support during the work 
sessions.

\bibliographystyle{alpha}
\bibliography{zombies}

\end{document}